\documentclass{QT}

\usepackage{epsfig,hhline,array}
\newtheorem{thm}{Theorem}[section]
\newtheorem{lemma}[thm]{Lemma}

\newtheorem{cor}[thm]{Corollary}

\theoremstyle{definition}
\newtheorem{df}[thm]{Definition}
\newtheorem{exam}[thm]{Example}

\DeclareMathOperator{\im}{im}

\def\S{\Sigma}
\def\cC{\mathcal C}
\def\cB{\mathcal B}
\def\cG{\mathcal G}
\def\cH{\mathcal H}
\def\R{{\mathbb R}}
\def\F{{\mathbb F}}
\newcommand\risS[6]{\raisebox{#1pt}[#5pt][#6pt]{\begin{picture}(#4,15)(0,0)
  \put(0,0){\includegraphics[width=#4pt]{#2.eps}} #3
     \end{picture}}}

\title[Polynomial invariants of graphs on surfaces]
  {Polynomial invariants of graphs on surfaces}
\author{R.~Askanazi, S.~Chmutov, C.~Estill, J.~Michel, P.~Stollenwerk}

\begin{document}
\maketitle

\address{Department of Mathematics, The Ohio State University,
231 West 18th Avenue, Columbus, OH 43210\hspace{\fill} \linebreak 
{\tt askanazi.3@osu.edu}, 
{\tt chmutov@math.ohio-state.edu},\hspace{\fill} \linebreak
{\tt estill@math.ohio-state.edu}, 
{\tt michel.82@osu.edu},
{\tt stollenwerk.2@osu.edu} }

%\subjclass[2010]{05C10, 05C31, 57M15, 57M25, 57M27}
\date{}


\begin{abstract}
For a graph embedded into a surface, we relate many combinatorial parameters of the cycle matroid of the graph and the bond matroid of the dual graph with the topological parameters of the embedding. This gives an expression of the polynomial, defined by M.~Las Vergnas in a combinatorial way using matroids as a specialization of the Krushkal polynomial, defined using the symplectic structure in the first homology group of the surface. 
\end{abstract}

\maketitle

\section*{Introduction} \label{s:intro}
Hassler Whitney introduced matroids in 1935 and gave two major examples of
matroids, the cycle matroid of a graph and a matroid defined by a finite collection of vectors in a vector space. 
Matroids have since found many applications in other parts of mathematics, including, in particular, topology 
where the matroidal properties of a hyperplane arrangement are very closely related to the topological properties of the arrangement (see \cite{BLSWZ,White,Yu,Zieg} and the references
therein).

In this paper we show that the matroids associated to graphs are intimately related to its topology

We consider graphs on surfaces.
Suppose that a graph $G$ is embedded into a surface $\S$ in a cellular manner; that is each connected component (face) of the complement to the graph is homeomorphic to a disc. Then we can define a dual graph $G^*$ embedded into the same surface $\S$ in a natural way.
We associate {\it cycle matroid} $\cC(G)$ with the graph $G$, and with the dual graph $G^*$ we associate its {\it bond matroid} $\cB(G^*)$ 
{\it dual} to $\cC(G^*)$.
In the planar case, when $\S$ is a sphere, $\cB(G^*)$ is isomorphic to $\cC(G)$.
Thus the difference between isomorphism  classes of $\cC(G)$ and $\cB(G^*)$ can be considered as a measure of non-planarity of the embedding and should reflect the topology of the pair $(\S,G)$.
Various combinatorial parameters of the matroids $\cC(G)$ and $\cB(G^*)$ can be assembled into the Las Vergnas polynomial of a matroid perspective $\cB(G^*)\to \cC(G)$ introduced in \cite{LV,LV2,LV3}.

On the other hand, the topological parameters of the embedding of $G$ into $\S$ can be assembled into the Krushkal polynomial introduced in \cite{Kru}.
In this paper we show that the Las Vergnas polynomial is a specialization of the Krushkal polynomial.
In the process, we relate many combinatorial parameters of the matroids $\cC(G)$ and $\cB(G^*)$ with the topological parameters of the embedding.

In Section \ref{s:LV} we briefly review matroids and their combinatorial parameters, and introduce the Las Vergnas polynomial.
We introduce the Krushkal polynomial in Section \ref{s:Kp}.
The main theorem is formulated in Section \ref{s:Th}, wherein we also obtain the duality property of the Las Vergnas polynomial as a consequence of our main theorem.
We begin to prove the main theorem in Section \ref{s:Th} and finish in Section \ref{s:mct} which consists of three lemmas relating the topological parameters of the embedding $G\hookrightarrow \S$ with the combinatorial parameters of its matroids.
We conclude in Section \ref{s:Kp-BR} with a discussion on the relation of the Krushkal polynomial with the Bollob\'as-Riordan polynomial of ribbon graphs.

%\section*{Acknowledgments}
This work has been done as a part of the Summer 2010 undergraduate research working group
\begin{center}\verb#http://www.math.ohio-state.edu/~chmutov/wor-gr-su10/wor-gr.htm#
\end{center}
``Knots and Graphs" 
at the Ohio State University.
We are grateful to all participants of the group for valuable discussions and to the OSU Honors Program Research Fund for the student financial support.
We thank the anonymous referee and editors for various suggestions improving the exposition of the paper.

\section{Matroids and the Las Vergnas polynomial} \label{s:LV}

For additional background on matroids we refer to \cite{Ox,Wel}, in addition to Whitney's classical paper \cite{Whi}.

\begin{df}
{\it A matroid} is a finite set $M$ with a {\it rank}
function $r$ that assigns a number to a subset of $M$ and
satisfies the following axioms:
\begin{itemize}
\item[({\bf R1})] {\it The rank of an empty subset is zero.}
\item[({\bf R2})] {\it For any subset $H\subset M$ and any element $y\not\in H$,
$$r(H\cup\{y\})=
   \left\{\begin{array}{l} r(H) \mbox\ ,\qquad or\\ r(H)+1\ .\end{array}\right.$$}
\item[({\bf R3})] {\it For any subset $H$ and two elements $y$,$z$ not in $H$,
if\ $r(H\cup y)=r(H\cup z)=r(H)$, then $r(H\cup\{y, z\})=r(H)$.}
\end{itemize}
\end{df}

There are two major examples of matroids we will be focusing on, although these examples do not exhaust the collection of matroids.
That is to say, there are matroids that do not represent either of the following situations.

The first example is the cycle matroid $\cC(G)$ of a graph $G$.
The underlying set $M$ is the set of edges $E(G)$ and the rank function is given by $r(H):=v(G)-c(H)$, where $v(G)$ is the number of vertices of $G$ and $c(H)$ is the number of connected components of the spanning subgraph of $G$ consisting of all the vertices of $G$ and edges of $H$.

The second example is a finite set of vectors in a vector space.
We may think about them as column vectors of a matrix.
The rank function is the dimension of the subspace spanned by the subset of vectors, or the rank of the corresponding submatrix. This example generalizes the first one when the matrix (over $\F_2$) is the incidence matrix of $G$, i.e. the matrix of the simplicial boundary map from the set of edges of $G$ to the set of vertices of $G$.

A subset $H\subset M$ is called {\it independent} if $r(H)=|H|$.
In the first example, the independent subsets are those subsets of edges which do not contain cycles.
In the second, independent subsets correspond to linearly independent subsets of vectors.
A {\it base} of a matroid is a maximal independent set.

A subset $H\subset M$ is called a {\it circuit} if $r(H)=|H|-1$.
In the first example, this new notion of a circuit and the traditional notion of a circuit in a graph match.
In the second example, a circuit is a subset of vectors with precisely one linear relation between them.

Given any matroid $M$, there is a dual matroid $M^*$ with the same underlying set and with the rank function given by $r_{M^*}(H):=|H|+r_M(M\setminus H)-r(M)$.
In particular $r(M)+r(M^*)=|M|$. Any base of $M^*$ is a complement to a base of $M$. 

The dual matroid to the cycle matroid of a graph $G$ is called the {\it bond matroid} of $G$: $\cB(G):=(\cC(G))^*$.
The circuits of $\cB(G)$ are the minimal edge cuts, also known as the {\it bonds} of $G$.
These are minimal collections of the edges of $G$ which, when removed from $G$, increase the number of connected components.
The Whitney planarity criteria \cite{Whi} says that a graph $G$ is planar if and only if its bond matroid $\cB(G)$ is the cycle matroid of some graph.
In this case, it will be the cycle matroid of the dual graph, $\cB(G)=(\cC(G))^*=\cC(G^*)$.

\bigskip
\begin{df}[\cite{LV,LV2,LV3}]
For two matroids $M$ and $M'$, a bijection $M\to M'$  is called a {\it matroid perspective} if any circuit of $M$ is mapped to a union of circuits of $M'$.
Equivalently, 
$$r_M(X)-r_M(Y) \geqslant r_{M'}(X)-r_{M'}(Y)
\mbox{\qquad for all\quad } Y\subseteq X\ ,
$$
where $r_M$ and $r_{M'}$ are the rank functions of matroids $M$ an $M'$.
\end{df}

\begin{df}[\cite{LV,LV2,LV3}]
The {\it Tutte polynomial} of a matroid perspective $M\to M'$ is the polynomial in variables $x,y,z$ defined by 
$$
%\begin{equation}\label{lasV}
T_{M\to M'}\ := \sum_{H\subseteq M} (x-1)^{r(M')-r_{M'}(H)} (y-1)^{n_M(H)}  
  z^{(r(M)-r_M(H))-(r(M')-r_{M'}(H))}\ ,
%\end{equation}
$$
where $n_M(H):=|H|-r_M(H)$ is the nullity in $M$.
\end{df}

{\bf Properties.} The usual Tutte polynomial of matroids $M$ and $M'$ can be recovered from the Tutte polynomial of matroid perspective in the following ways:\\
$\begin{array}{l}
T_M(x,y) = T_{M\to M}(x,y,z)\ ;\\
T_M(x,y) = T_{M\to M'}(x,y,x-1)\ ; \\
T_{M'}(x,y) = (y-1)^{r(M)-r(M')} T_{M\to M'}(x,y,\frac1{y-1})\ . 
\end{array}$

\bigskip
For graphs $G$ and $G^*$ dually cellularly embedded in a surface $\S$, the natural map of the bond matroid of $G^*$ onto the cycle matroid of $G$, $\cB(G^{*})\to \cC(G)$, is a matroid perspective. Formally this follows from
a theorem of J.~Edmonds \cite{Ed}. Informally, a circuit $c$ of $\cB(G^{*})$
is a minimal cut of the dual graph $G^*$ which separates the vertices of $G^*$ into two sets. The vertices of $G^*$ correspond to the faces of the original graph $G$. Thus, the minimal cut separates the faces of $G$.
We may think about cutting an edge of $c$ as cutting the surface $\S$ along the corresponding edge of $G$. Therefore, in terms of the original graph $G$, the circuit $c$ corresponds to a subset of edges of $G$ which separate the surface $\S$. Topologically it represents a zero-homologous cycle which (in general) consists of several circuits. Hence, a circuit of $\cB(G^{*})$ is mapped to a union of circuits of $\cC(G)$.

We call the Tutte polynomial $T_{\cB(G^{*})\to \cC(G)}(x,y,z)$ 
of the matroid perspective 
$\cB(G^{*})\to \cC(G)$ the {\it Las Vergnas polynomial} of the graph $G$ on the surface $\S$, and denote it $LV_{G,\S}(x,y,z)$. 
One goal of this paper is to give a topological interpretation of various combinatorial ingredients of this polynomial.

In this paper we assume that $\S$ is orientable.

\begin{exam}\label{ex1}
Let $G$ be a graph with one vertex and two loops embedded into a torus as shown. Then $G^*$ is the similar graph.
$$\risS{-20}{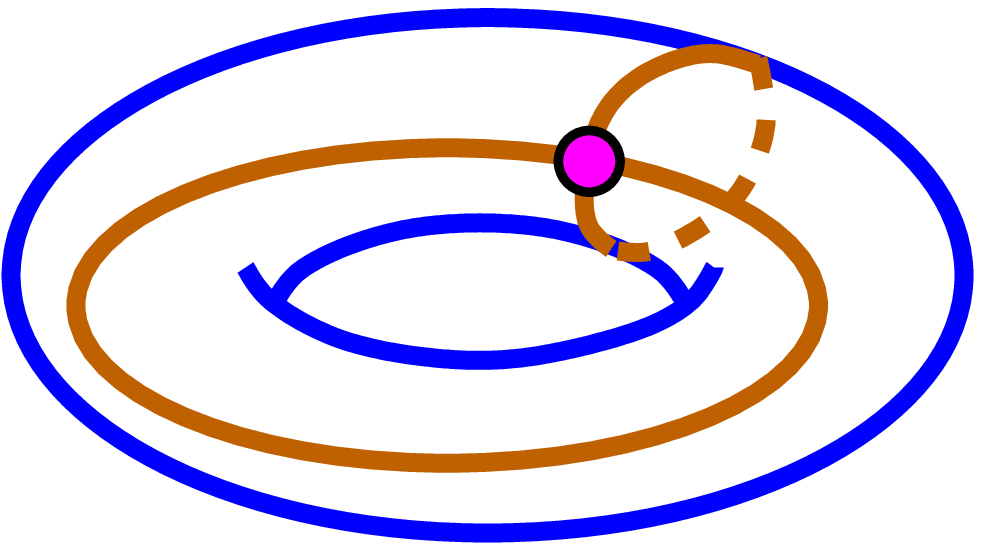}{\put(80,50){$G^*$}}{100}{35}{25}\hspace{2cm}
\risS{-20}{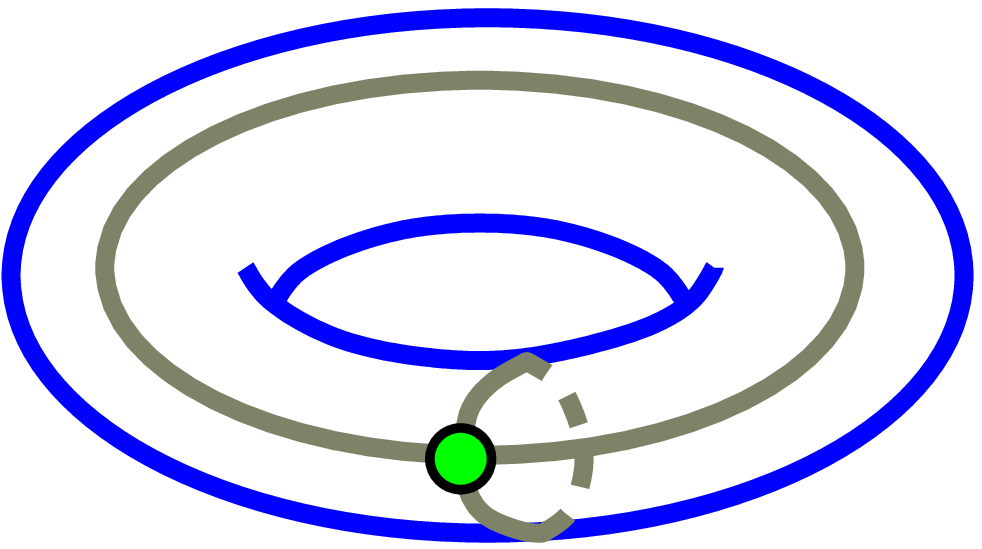}{\put(0,45){$G$}}{100}{0}{0}
$$
In this case the bond matroid $M=\cB(G^*)$ has rank 2, and the cycle matroid $M'=\cC(G)$ has rank 0. 
For any subset $H$: $r_M(H)=|H|$, $n_M(H)=0$, and $r_{M'}(H)=0$. We have
$$LV_{G,\S}(x,y,z)=z^2+2z+1\ .$$
\end{exam}

\section{Krushkal polynomial} \label{s:Kp}

V. Krushkal discovered his polynomial while researching topological quantum field theories and the algebraic and combinatorial properties of models of statistical mechanics.  This polynomial can be seen as a generalization of the Tutte and Bollobas-Riordan polynomials \cite{Kru}.

\begin{df}[\cite{Kru}]
For a graph $G$ embedded into a (not necessarily connected, but orientable) surface $\S$,
\begin{equation}\label{krush}
P_{G,\S}(X, Y, A, B) := \sum_{H\subseteq G} X^{c(H)-c(G)} Y^{k(H)} A^{s(H)/2}  
   B^{s^\perp(H)/2},
\end{equation}
where\vspace{-5pt}
\begin{itemize}
\item $c(H)$ is the number of connected components of the spanning subgraph 
%with the edge set 
$H$; 
\item the restriction of the embedding $G\hookrightarrow\S$ to $H$ induces a map on the first homology groups and we define
$$k(H):= \dim(\ker(H_1(H;\R) \to H_1(\S;\R)))\ ;$$
\item $s(H)$ is equal to twice the genus of a regular neighborhood of the spanning subgraph $H$ in $\S$ (the neighborhood is a surface with boundary, and its genus is defined as the genus of the closed surface obtained by attaching a disk to each boundary circle); 
\item $s^\perp(H)$ is equal to twice the genus of the surface obtained by removing a regular neighborhood of $H$ from $\S$.
\end{itemize}
\end{df}

\bigskip
{\bf Remark.} V.~Krushkal indicates in his paper \cite{Kru} that the parameters $s(H)$ and $s^\perp(H)$ have the following interpretation in terms of the symplectic bilinear form on the vector space $H_1(\S;\R)$ given by the intersection number.
For a given spanning subgraph $H$, let $V$ be its image in the homology group:
$$H_1(\S;\R)\supset V := V(H) := \im(H_1(H;\R) \to H_1(\S;\R))\ .
$$
For the subspace $V$ we can define its orthogonal complement $V^\perp$ in $H_1(\S;\R)$ with respect to the symplectic intersection form. Then
$$s(H)=\dim(V/(V\cap V^\perp))\ ,\qquad s^\perp(H)=\dim(V^\perp/(V\cap V^\perp))\ .
$$

\begin{exam}\label{ex2}
Continuing with example \ref{ex1}, for the graph $G$ on the torus $\S$, the map $H_1(G;\R) \to H_1(\S;\R)$ is not degenerate. Therefore, $k(H)=0$ for any subset $H$.
Also, $c(H)=1$ for any $H$.
If $H\not=G$, then its regular neighborhood has genus 0. Hence, for such $H$, $s(H)=0$, while $s(G)=2$.
Similarly, if $H\not=\emptyset$, then the regular neighborhood of its complement also has genus 0. So, for such $H$, $s^\perp(H)=0$, while $s^\perp(\emptyset)=2$.
Combining all this we get
$$P_{G,\S}(X, Y, A, B) = B +2+A\ .
$$

\end{exam}

\section{Main Theorem} \label{s:Th}

\begin{thm}
Suppose $G$ is cellularly embedded in an orientable surface $\S$ of genus $g$.
Then
\begin{equation}\label{subst}
   LV_{G,\S}(x,y,z) = z^g P_{G,\S}(x-1, y-1, z^{-1}, z).
\end{equation}
\end{thm}

\begin{proof}
The summands corresponding to each subgraph $H$ of $G$ will be shown to be equal.
Applying the substitution \eqref{subst} in \eqref{krush} we arrive at summands of the form
$$(x-1)^{c(H)-c(G)} (y-1)^{k(H)} z^{g-s(H)/2+s^\perp(H)/2}.
$$

The cycle matroid rank of a graph is given by $r(M') = v(G)-c(G)$ and that of a subgraph $H \subseteq G$ by $r_{M'}(H) = v(G) - c(H)$.
So $r(M')-r_{M'}(H) = c(H)-c(G)$ and hence the powers of the $(x-1)$ factor coincide.

It remains to prove the equality between the exponents of $(y-1)$ and $z$:
\begin{itemize}
\item $k(H)=n_M(H)$
\item $g-s(H)/2+s^\perp(H)/2=r_M(G)-r_{M'}(G)-r_M(H)+r_{M'}(H)$
\end{itemize}
These will be proved separately in the next section.
\end{proof}

\begin{cor}[\cite{LV,LV2,LV3}]
$$LV_{G^*,\S}(x,y,z) = z^{2g} LV_{G,\S}(y,x,z^{-1})\ .
$$
\end{cor}
This follows from Krushkal's formula \cite{Kru}:  $P_{G^*,\S}(X, Y, A, B) = P_{G,\S}(Y, X, B, A)$ for cellular embeddings $G\hookrightarrow \S$.
   
\section{Matroidal combinatorics and combinatorial topology} \label{s:mct}

In this section we relate the topological parameters of the embedding $G\hookrightarrow \S$ with the combinatorial parameters of the matroid perspective $\cB(G^{*})\to \cC(G)$.

\begin{lemma}\label{l:4-1}
$k(H)=n_M(H)$.
\end{lemma}

\begin{proof}
Let $N:=\cC(G^*)=(\cB(G^*))^*=M^*$.
The rank of a dual matroid can be defined in terms of the rank of a matroid by $r_M(H) = |H|+r_N(E \setminus H)-r_N(N) $, where $H$ is a subset of edges of $G^*$ that can be naturally identified with the corresponding subset of edges of $G$.
Since the nullity is defined as the number of edges minus the rank we have that $n_M(H) = r_N(N) - r_N(G^* \setminus H)$, where $G^* \setminus H$ is the spanning subgraph of $G^*$ consisting of the edges not in $H$ and all vertices of $G^*$.
Thus, because $N$ is the cycle matroid of the graph $G^*$, we have 
$$n_M(H) = (v(G^*)-c(G^*)) - (v(G^*)-c(G^*\setminus H))= c(G^*\setminus H) - c(G^*)\ .
$$

Obviously, $c(G^*)=c(G)$ and is equal to the number of connected component $c(\S)$ of $\S$.
   
Now we consider $H$ as a spanning subgraph of $G$.
We want to remove its regular neighborhood from the surface $\S$ and count the number of connected components $c(\S \setminus H)$.
First we remove small discs around all vertices of $H$, i.e. all vertices of $G$.
These are exactly the faces of $G^*$.
So we are left with a regular neighborhood of $G^*$ in $\S$. 
Secondly, removing neighborhoods of the edges of $H\subset G$ from $\S$ will give us the same surface, topologically, as deleting the corresponding neighborhoods of the edges of $H\subset G^*$, because these edges are transverse to each other.
In other words $c(G^*\setminus H)$ is the number of components of $\S \setminus H$.
Hence 
$$n_M(H) = c(\S\setminus H) - c(G)\ ,$$
where $H$ is regarded as a spanning subgraph of $G$.

   Let us turn our attention to the number that we wish to show $n_M(H)$ to be equal to.
Denote by $i_*:H_1(H;\R) \to H_1(\S;\R)$ the linear map induced by the composition of embeddings $H\hookrightarrow G\hookrightarrow \S$.
We have $k(H) = \dim(\ker(i_*))$.

The topological pair $(\S, H)$ gives us a long exact sequence of homology groups
$$\cdots \to H_2(H) \to H_2(\S) \to H_2(\S, H) \xrightarrow{\delta} H_1(H) \xrightarrow{i_*} H_1(\S) \to \cdots\ .
$$
$H_2(H)$ is trivial as $H$ is one dimensional,
$H_2(\S)$ has dimension equal to the number of components of $\S$.
and $H_2(\S,H)$ has dimension equal to the number of components of $\S\setminus H$.

So if we turn our attention to the short exact sequence
$$0 \to \R^{c(\S)} \to H_2(\S, H) \to \im \delta \to 0\ ,
$$
wherein $\im \delta = \ker i_*$, we see that 
$\dim \ker i_*= c(\S\setminus H) - c(\S) = c(\S\setminus H) - c(G)=n_M(H)$.
 \end{proof}

Thus the $(y-1)$ powers in the main theorem coincide.

In fact, the argument with the long exact sequence was used by V.~Krushkal 
\cite[end of proof of Theorem 3.1]{Kru} where he essentially proved that
$k(H)=c(\S\setminus H) - c(\S)$ using a slightly different terminology.
However, the relation of this parameter to the matroidal $n_M(H)$ was not addressed there.

\bigskip
\begin{lemma}\label{l:4-2}
$2g=r_M(G)-r_{M'}(G)$.
\end{lemma}

\begin{proof}
$r_{M'}(G)=v(G)-c(G)=e(T)$, where $T$ is a spanning forest of $G$.
The bond rank of a graph $G^*$ is the maximal number of edges that one can delete from it without increasing the number of connected components, i.e.\ all edges but a spanning forest, $T^*$, of the graph $G^*$.
So $r_M(G) = e(G^*) - e(T^*) = e(G) - e(T^*)$ and $r_M(G)-r_{M'}(G) = e(G) - e(T^*) - e(T)$.
But the number of edges in a spanning forest of $G$ is equal to the number of vertices minus the number of components.
Similarly, $e(T^*)$, the number of edges in a spanning forest of $G^*$, is the number of faces, $f$, of $G$ minus the number of components.
So 
$r_M(G)-r_{M'}(G) = e - (f-c(\S)) - (v-c(\S)) = e-f-v+2c(\S) = 2g(\S) = 2g$.
\end{proof}

\bigskip
\begin{lemma}\label{l:4-3}
$g+s(H)/2-s^\perp(H)/2=r_M(H)-r_{M'}(H)$.
\end{lemma}

\begin{proof}
Consider $r_M(H)-r_{M'}(H)$.
For either matroid, $M$ or $M'$, rank is equal to the number of edges minus nullity.
So $r_M(H)-r_{M'}(H) = n_{M'}(H) - n_M(H)$.
We have seen in Lemma \ref{l:4-1} that $n_M(H)=k(H)$. Formula (2.5) in Krushkal's paper \cite{Kru} tells us that this nullity 
$$n_{M'}(H) = k(H) + g + s(H)/2 - s^\perp(H)/2\ .
$$
So we have $r_M(H)-r_{M'}(H)$ equaling
$$k(H) + g + \frac{s(H)}{2} - \frac{s^\perp(H)}{2} - k(H)
= g + \frac{s(H)}{2} - \frac{s^\perp(H)}{2}\ .
$$
\end{proof}

Lemma \ref{l:4-2} is a particular case of Lemma \ref{l:4-3} where we take $H=G$.
Indeed, in this case  we have $s(G)=2g$ and $s^\perp(G)=0$.

\bigskip
Lemmas \ref{l:4-2} and \ref{l:4-3} together imply
$$g-s(H)/2+s^\perp(H)/2=r_M(G)-r_{M'}(G)-r_M(H)+r_{M'}(H)\ .
$$
Thus the $z$ powers in the main theorem coincide. This completes the proof of the main theorem.

\section{Krushkal and Bollob\'as-Riordan polynomials} \label{s:Kp-BR}

A cellular embedding $G\hookrightarrow \S$ may be studied in terms of a {\it ribbon graph} which represents a regular neighborhood $\cG$ of $G$ in $\S$.
Working backwards, starting with a ribbon graph $\cG$, we can construct a surface $\S$ by capping all boundary components of $\cG$ by discs.
Then the core graph $G$ of $\cG$, obtained by contracting all edge-ribbons to their central lines and all vertex-discs to their central points, can be cellularly embedded in $\S$.
For ribbon graphs, we have the Bollob\'as-Riordan polynomial \cite{BR3} defined as
$$BR_\cG(X,Y,Z):=\sum_{\cH\subseteq \cG} 
(X-1)^{c(\cH)-c(\cG)}Y^{n(\cH)}Z^{c(\cH)-bc(\cH)+n(\cH)}\ ,
$$
where $bc(\cH)$ is the number of boundary components of the spanning ribbon subgraph $\cH$. Note that the exponent $c(\cH)-bc(\cH)+n(\cH)$ is equal to $2g(\cH)$ for oriented ribbon graphs.
 
V.~Krushkal proved \cite[Lemma 4.1]{Kru} that
\begin{equation}\label{eq:P2BR}
BR_\cG(X,Y,Z)=Y^g P_{G,\S}(X-1,Y,YZ^2,Y^{-1})\ .
\end{equation}
 
It was proved in \cite[Theorem 2]{BR3} that $BR_\cG$ is universal in the class of polynomials satisfying the contraction/deletion property.
The Krushkal polynomial also satisfies \emph{a} contraction/deletion property \cite[Lemma 2.1]{Kru}.
Based on that V.~Krushkal wrote that $BR_\cG$ and $P_{G,\S}$ carry equivalent information.
However this is not the case as the contraction/deletion properties for $BR_\cG$ and for $P_{G,\S}$ are not quite the same.
The problem arises when deletion of an edge of a ribbon graph changes its genus.
The genus might decrease by 1 with removal of an edge.
For example, if we delete a loop $e$ from the ribbon graph $\cG$ corresponding to $G$ from Example \ref{ex1}, then the resulting graph with a single loop will have genus zero.
So, while in the Bollob\'as-Riordan approach it is considered as a graph embedded into a sphere, in the Krushkal approach it is still embedded into the torus.
We cannot apply the substitution \eqref{eq:P2BR} to that graph
since its embedding on the torus is no longer cellular.
Thus the Krushkal polynomial does not satisfy the contraction/deletion property in the sense of Bollob\'as and Riordan.

We also find that the Las Vergnas polynomial $LV_{G,\S}(x,y,z)$ does not satisfy the contraction/deletion property in the sense of Bollob\'as and Riordan either.

\begin{exam}\label{ex3}
This is an example of a calculation of the three polynomials.
% showing that
%the Krushkal and the Las Vergnas polynomials do not satisfy the contraction/deletion property in the sense of Bollob\'as and Riordan.
Here $G$ is a graph on torus with two vertices and three edges $a$, $b$, and $c$. Its dual $G^*$ has one vertex and three loops. The ribbon graph corresponding to $G$ is denoted $\cG$. We use the same symbols $a$, $b$, $c$
to denote the corresponding edges in all three graphs.
$$G^*=\risS{-25}{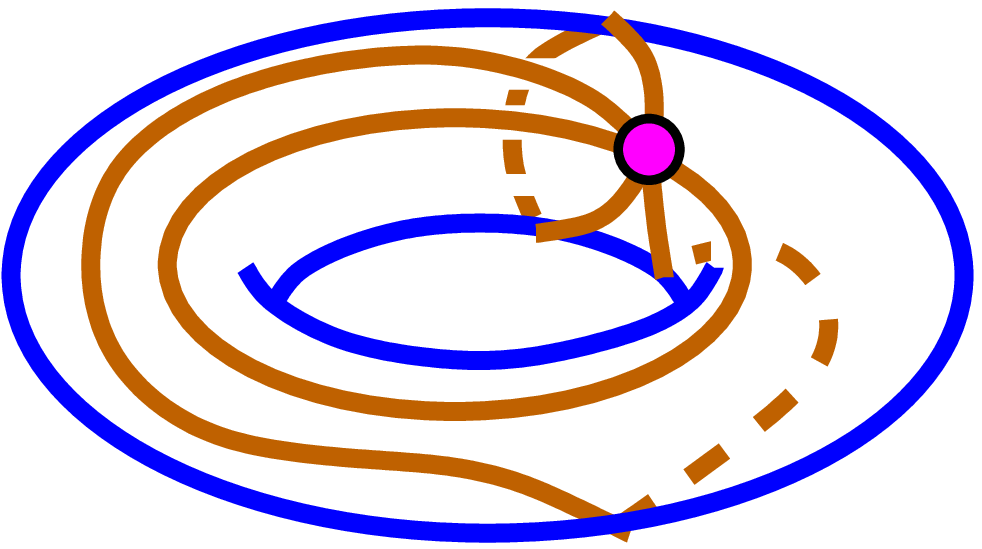}{\put(65,55){$a$}\put(35,34){$b$}
                        \put(58,6){$c$}}{100}{35}{35}\ \hspace{1cm}
G=\risS{-25}{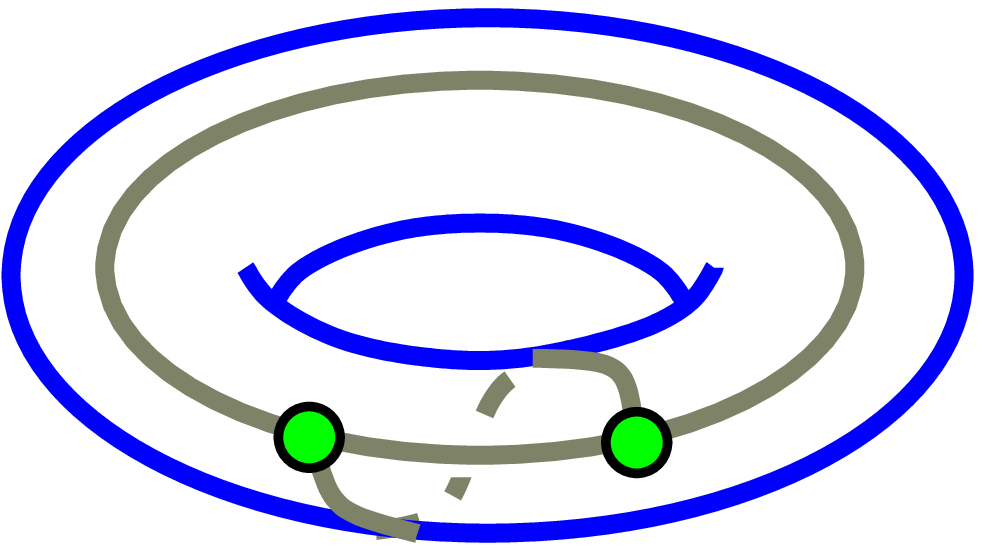}{\put(70,36){$a$}\put(35,-6){$b$}
                        \put(55,4){$c$}}{100}{35}{35}\ %\hspace{1cm}
$$
$$
\cG=\risS{-22}{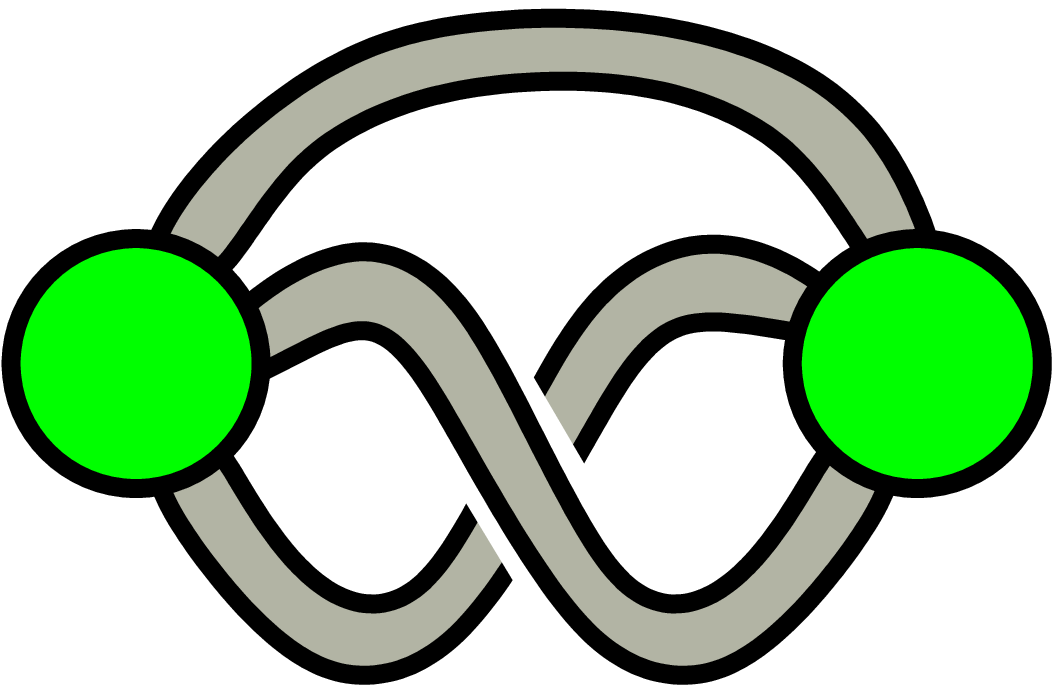}{\put(67,45){$a$}\put(15,-3){$b$}
                        \put(55,-4){$c$}}{80}{0}{0}
$$
The matroid $M'=\cC(G)$ is of rank 1, and for any nonempty subset $H$, $r_{M'}(H)=1$. The cycle matroid $\cC(G^*)$ of the dual graph is of rank zero because $G^*$ has only loops. So its dual $M=\cB(G^*)$ has rank 3, all subsets $H$ are independent and $r_M(H)=|H|$. The next table shows the value of various parameters and contributions of all eight subsets $H\subseteq\{a,b,c\}$ to the three polynomials.
$$\begin{array}{cr||!{\ \makebox(0,10){}}c|c|c|c|c|c|c|c}
 &H& \emptyset&\{a\}&\{b\}&\{a,b\} &\{c\}&\{a,c\}&\{b,c\}&\{a,b,c\}\!\!\!  
        \\[2pt]
\hline\hline
 &c(H)& 2&1&1&1 &1&1&1&1 \\ \hhline{~---------}
 &k(H)& 0&0&0&0 &0&0&0&0 \\ \hhline{~---------}
 &s(H)& 0&0&0&0 &0&0&0&2 \\ \hhline{~---------}
 &s^\perp(H)& 2&2&2&0 &2&0&0&0 \\ \hhline{~---------}
\makebox(20,0){\raisebox{60pt}{\rotatebox{90}{Krushkal}}}
 &P_{G,\S}& XB&B&B&1 &B&1&1&A \\ \hline\hline

 &r_M(H)& 0&1&1&2 &1&2&2&3 \\ \hhline{~---------}
 &r_{M'}(H)& 0&1&1&1 &1&1&1&1 \\ \hhline{~---------}
 &n_M(H)& 0&0&0&0 &0&0&0&0 \\ \hhline{~---------}
\makebox(20,0){\raisebox{50pt}{\rotatebox{90}{Las Vergnas}}}
 &LV_{G,\S}& \!\!\!(x-1)z^2&z^2&z^2&z &z^2&z&z&1 \\ \hline\hline

 &c(\cH)& 2&1&1&1 &1&1&1&1 \\ \hhline{~---------}
 &n(\cH)& 0&0&0&1 &0&1&1&2 \\ \hhline{~---------}
 &bc(\cH)& 2&1&1&2 &1&2&2&1 \\ \hhline{~---------}
\makebox(20,0){\raisebox{50pt}{\rotatebox{90}{
              $\genfrac{}{}{0pt}{0}{\mbox{Bollob\'as}}{\mbox{Riordan}}$}}}
 &BR_\cG& (X-1)&1&1&Y &1&Y&Y&Y^2Z^2 \\ 
\end{array}
$$
Thus
$$P_{G,\S}=3+3B+XB +A,\qquad
LV_{G,\S}=3z+3z^2+(x-1)z^2+1,
$$
$$
BR_\cG=3+3Y+(X-1)+Y^2Z^2.
$$
One can readily confirm the relations \eqref{subst} and \eqref{eq:P2BR} from here.

Now if we contract the edge $c$, the graph $G/c$ still will be cellularly  embedded into the same torus $\S$, and its regular neighborhood coincides with the ribbon graph $\cG/c$.
Examples \ref{ex1} and \ref{ex2} and the right part of the table above give the following polynomials:
$$P_{G/c,\S}=B+2+A,\qquad
LV_{G/c,\S}=z^2+2z+1,\qquad
BR_{\cG/c}=1+2Y+Y^2Z^2.
$$
Meanwhile if we delete the edge $c$, then 
$$P_{G-c,\S}=XB+2B+1\ .
$$
But the graph $G-c$ is not cellularly embedded into the torus $\S$ any more. Thus the Las Vergnas and the Bollob\'as-Riordan polynomials are not defined for it. Its regular neighborhood gives the ribbon graph $\cG-c$ which, after capping the discs to its two boundary components, results in the sphere $S^2$. Thus the graph $G-c$ embeds cellularly into the the sphere $S^2$. For this embedding we have
$$P_{G-c,S^2}=X+2+Y,\qquad
LV_{G-c,S^2}=(x-1)+2+(y-1),\qquad
BR_{\cG-c}=(X-1)+2+Y.
$$
Therefore 
$$P_{G,\S}=P_{G-c,\S}+P_{G/c,\S}\qquad\mbox{and}\qquad
BR_{\cG}=BR_{\cG-c}+BR_{\cG/c}\ ,
$$
but
$$P_{G,\S}\not=P_{G-c,S^2}+P_{G/c,\S}\qquad\mbox{and}\qquad
LV_{G,\S}\not=LV_{G-c,S^2}+LV_{G/c,\S}\ .
$$
\end{exam}

\bigskip
Currently, according to relations \eqref{subst} and \eqref{eq:P2BR}, the Krushkal polynomial is the most general polynomial of graphs on surfaces and so it clearly deserves further research.
Also, because  the two relations look quite different, the Las Vergnas and the Bollob\'as-Riordan polynomials seem to be independent. 

Very recently the Krushkal polynomial was generalized to higher dimensional simplicial complexes \cite{KR}. It is related to a matroid 
%in a different way. Namely, to the matroids 
on the sets of simplices of the middle dimension for a triangulation of an even dimensional sphere, where a subset of simplices is indpendent if and only if they are mapped into the the linear independent chains by the simplicial boundary map. It turns out that the dual matroid corresponds to the dual triangulation and
the Tutte polynomial of this matroid corresponds to the ``higher dimensional" Tutte polynomial of the simplicial complexes (see \cite{KR} for details). 

\bigskip
%\newpage
{\bf Note added in proof.} After the paper has been submitted to the journal
Jonathan Michel found the following example of two ribbon graphs with the same
Las Vergnas polynomials but different Krushkal polynomials.
$$\cG_1=\ \risS{-30}{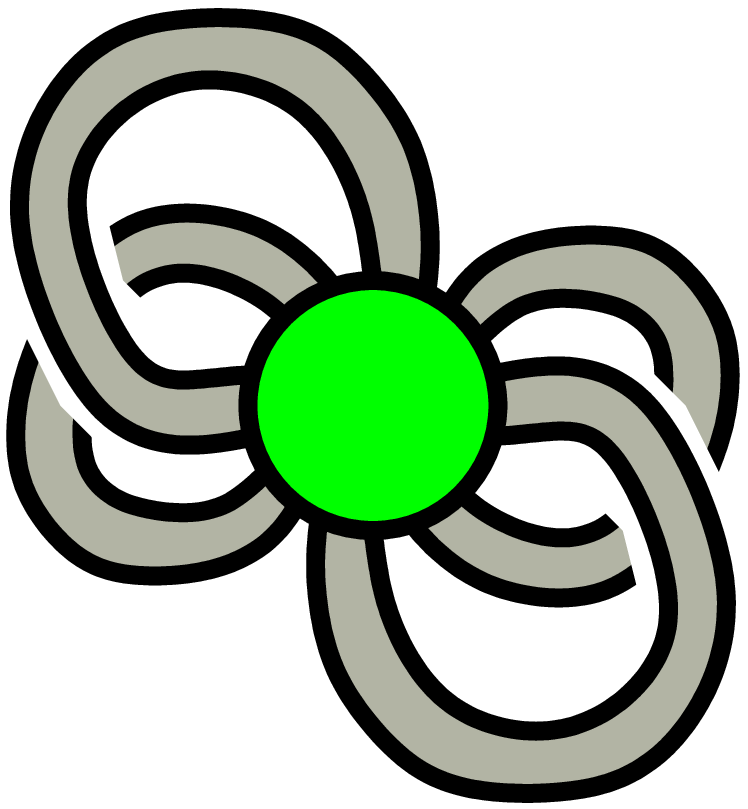}{}{60}{35}{35}\ \hspace{1cm}
\cG_2=\risS{-25}{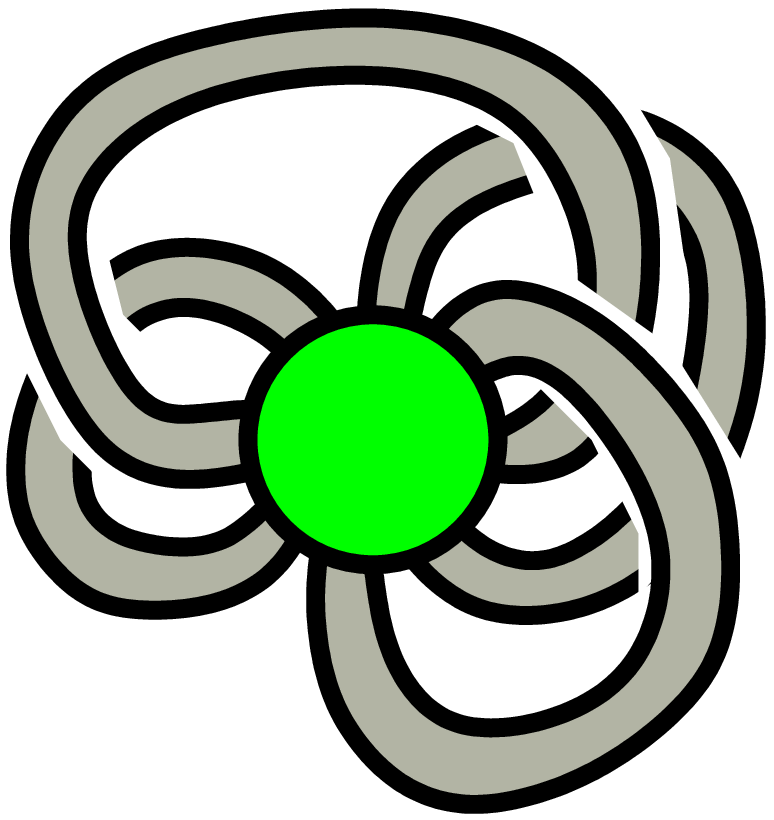}{}{60}{35}{35}\ %\hspace{1cm}
$$
Let $\S_1$ (resp. $\S_2$) be a surface obtained from $\cG_1$ (resp. $\cG_2$)
by gluing a disc to its boundary component, and $G_1\subset\S_1$ 
(resp. $G_2\subset\S_2$) be the corresponding core graph. Then
$LV_{G_1,\S_1}=LV_{G_2,\S_2}=(1+z)^4$. But
$$P_{G_1,\S_1}= A^2+4A+2AB+4+4B+B^2$$
and
$$P_{G_2,\S_2}= A^2+4A+4AB+2+4B+B^2\ .$$

Jonathan Michel also found two different ribbon graphs with the same  Krushkal polynomials.
$$\cG_3=\ \risS{-25}{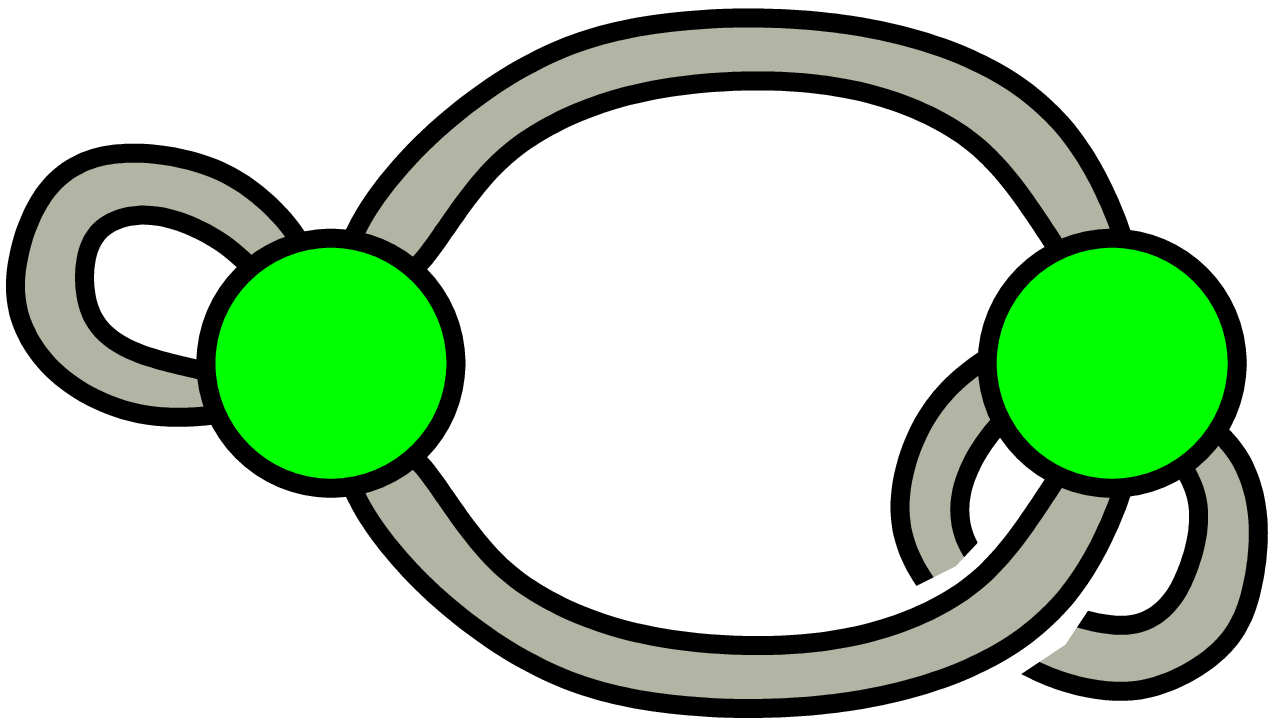}{}{100}{25}{35}\ \hspace{1cm}
\cG_4=\risS{-25}{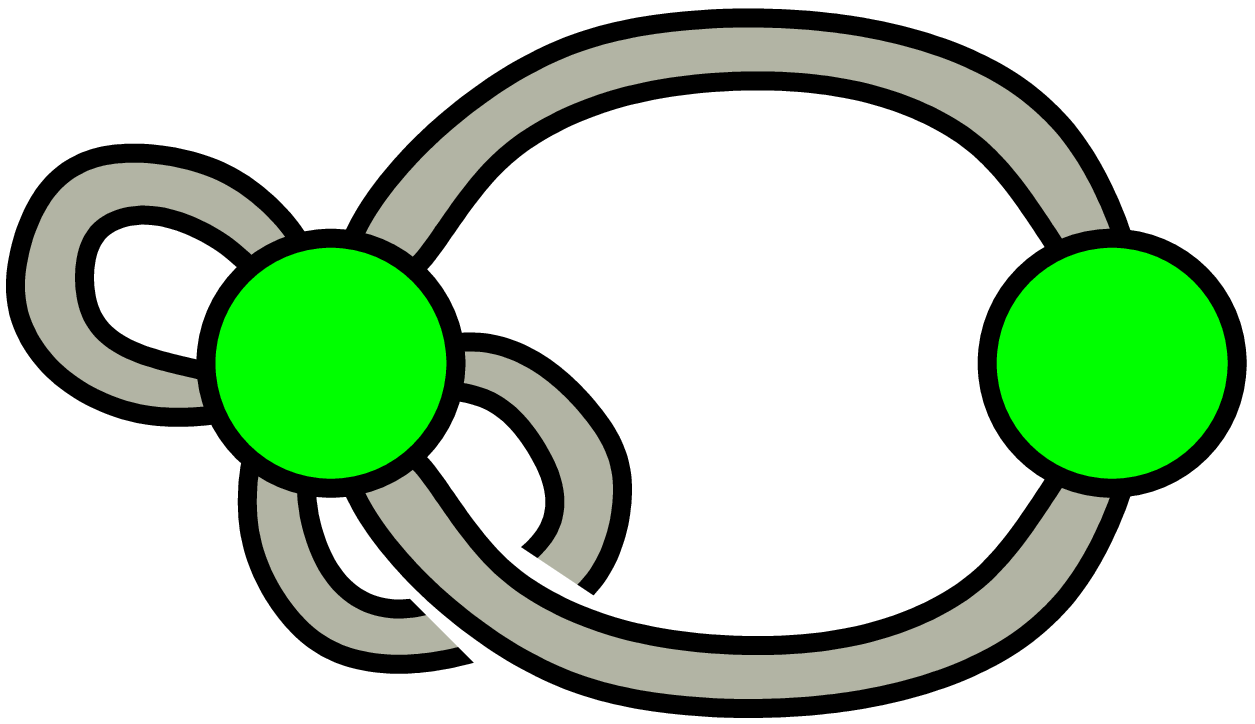}{}{100}{0}{0}\ %\hspace{1cm}
$$
For the corresponding core graphs $G_3\subset\S_3$ and $G_4\subset\S_4$ we have
$$P_{G_3,\S_3}= P_{G_4,\S_4}=YA+4Y+A+2YB+3+2B+XYB+X+XB^2\ .$$

\end{document}